\documentclass{amsart}

\usepackage{cgeometry}
\usepackage{cleveref}

\bibliography{ComplexGeometry}
\begin{document}

\title[Liu morphisms]{On Liu morphisms in non-Archimedean geometry}
\author{Mingchen Xia}
\date{\today}

\begin{abstract}
We define Liu morphisms and quasi-Liu morphisms between Berkovich analytic spaces. We show that Liu morphisms and quasi-Liu morphisms behave as affine morphisms and quasi-affine morphisms of schemes in many aspects. 
\end{abstract}

\maketitle

%\textcolor{red}{This is the first one of a series of two papers, the second paper is not yet finished. Feel free to contact me for comments! For the latest version, see \url{http://www.math.chalmers.se/~xiam/LM.pdf}.}

%\begin{abstract}
%\end{abstract}

\tableofcontents

%\newpage

\section{Introduction}

\subsection{Motivation}
In classical algebraic geometry, the theories of affine morphisms and quasi-affine morphisms play a prominent role. In the non-Archimedean world, it is highly desirable to have analogous results as well. However, there are two principal difficulties in the non-Archimedean setting: 
\begin{enumerate}
\item First of all, there is no satisfactory theory of quasi-coherent sheaves in non-Archimedean geometry. There is indeed an \emph{ad hoc} notion of quasi-coherent sheaves in rigid geometry defined by Conrad in \cite{Con06}: A quasi-coherent sheaf is a sheaf of modules which can be expressed as a filtered colimit of coherent sheaves locally. However, Conrad's notion of quasi-coherent sheaves does not behave as expected: On an affinoid space, the higher cohomologies of a quasi-coherent sheaf do not vanish in general. This makes it hard to handle affine morphisms in terms of quasi-coherent sheaves of algebras. The same problem persists in Berkovich geometry.
\item Secondly, a more severe problem was proposed by Liu \cite{Liu88}, \cite{Liu90}. It is shown that there is a quasi-compact, separated non-affinoid rigid space $X$, a morphism $f:X\rightarrow Y$ to an affinoid space $Y$, an admissible affinoid covering $\{U_i\}$ of $Y$ such that $f^{-1}U_i$ is affinoid for each $i$. See \cite[Proposition~3.3 and Section~5]{Liu90}. This means that the property that the inverse image of an affinoid domain is affinoid is \emph{not} G-local. 
\end{enumerate}
Recall that in classical algebraic geometry, we have the celebrated \emph{Serre's criterion} (\cite[Théorème~5.2.1]{EGAII}): Affine schemes can be characterized by cohomological triviality among quasi-compact separated schemes. Similarly, in non-Archimedean setting, we replace the usual local notion of affinoid spaces by cohomologically trivial spaces. Such spaces are studied by Maculan--Poineau in \cite{MP21} under the name of \emph{Liu spaces}, we follow their terminology. 
\begin{definition}[c.f. \cref{def:Liuspace}]
Let $k$ be a complete non-Archimedean valued field. A quasi-compact, separated $k$-analytic space $X$ (in the sense of Berkovich) is said to be \emph{Liu} if for any analytic extension $k'/k$, any coherent sheaf $\mathcal{F}$ on $X_{k'}$ is acyclic.
\end{definition}

On the morphism level, we define a \emph{Liu morphism} as a morphism under which the inverse image of a Liu domain is a Liu space, see \cref{def:Liumorp}.  
Similarly, we have a notion of quasi-Liu morphisms analogous to the classical notion of quasi-affine morphisms:
\begin{definition}[c.f. \cref{def:quasiLiumor}]
Let $f:X\rightarrow Y$ be a morphism of $k$-analytic spaces. We say $f$ is \emph{quasi-Liu} if for any Liu domain $Z$ in $Y$, $f^{-1}Z$ can be embedded in a Liu $k$-analytic space as a compact analytic domain and $H^0(f^{-1}Z,\mathcal{O}_X)$ is a Liu $k$-algebra (\cref{def:Liualge}).
\end{definition}

Similar to the situation in classical algebraic geometry, we prove a cohomological criterion of Liu morphisms when $Y$ is Liu in \cref{thm:Liumoreq}.

\textcolor{red}
{Unfortunately, as pointed out by Marco Maculan, contrary to the assertion in the previous versions of this paper, the notion of Liu morphisms is not G-local on the target, see an example due to Scholze--Weinstein in \cref{ex:notGlocal}.}

As for (1), due to the progress made by Ben-Bassat--Kremnizer in \cite{BBK17}, it is by far clear that the natural notion on a non-Archimedean analytic space is not that of the quasi-coherent sheaves, but the derived category of quasi-coherent sheaves instead. However, as we will see, in the special case of sheaves of Liu algebras studied below, the derived notion reduces to a \emph{bona fide} notion of quasi-coherence at the non-derived level. In particular, on a separated space, there is a global notion of quasi-coherent sheaves of Liu algebras, see \cref{def:qcohshLiualg}.

\subsection{Main results}

We fix a complete non-Archimedean valued field $k$. We allow the valuation on $k$ to be trivial. We work in the framework of Berkovich spaces as in \cite{Berk93}.

The main result says that Liu morphisms and quasi-coherent sheaves of Liu $k$-algebras are essentially equivalent:
\begin{theorem}[=\cref{cor:Liuequivalence}]
Let $X$ be a separated $k$-analytic space. Then the functor 
\[
	\Specrel_X: \QcohLiuAlgCat_{X,k} \rightarrow \LiuMorCat_{\rightarrow X,k}
\]
is an anti-equivalence of categories.
\end{theorem}
Here $\QcohLiuAlgCat_{X,k}$ is the category of quasi-coherent sheaves of Liu $k$-algebras on $X$, $\LiuMorCat_{\rightarrow X,k}$ is the category of Liu morphisms $Y\rightarrow X$. The functor $\Specrel_X$ is the relative spectrum functor defined in \cref{def:relspec}. This result is analogous to the classical result on affine morphisms and quasi-coherent sheaves of algebras (\cite[Proposition~1.2.7, Proposition~1.3.1]{EGAII}).

\subsection{Structure of the paper}
In \cref{sec:pre}, we recall some basic results about Berkovich analytic spaces and the language developed by Ben-Bassat and Kremnizer (\cite{BBK17}). Due to the lack of references, we also prove a representability theorem (\cref{thm:rep}) about presheaves on the category of analytic spaces.

In \cref{sec:Liusp}, we recall the basic theory of Liu spaces and Liu algebras. We prove that Liu algebras behave very similar to affinoid algebras in many aspects.

In \cref{sec:Liumor}, we introduce Liu morphisms and study their relation to quasi-coherent sheaves of Liu algebras.

In \cref{sec:quasiLiu}, we introduce and study quasi-Liu morphisms.

%In \cref{sec:fpqc}, we prove that Liu morphisms and quasi-Liu morphisms can be effectively descent with respect to a fpqc covering.

In \cref{sec:open}, we give a list of unsolved problems related to this work.

We collect results from \cite{BBK17} in \cref{sec:BBK}. 

\subsection{Conventions}

Let $k$ be a complete non-Archimedean valued field. An \emph{analytic extension} of $k$ is a complete non-Archimedean valued field $k'$ containing $k$ such that the restriction of the valuation on $k'$ to $k$ coincides with the given valuation on $k$. We denote the spectrum of a Banach algebra $A$ by $\Sp A$ instead of the more common notation $\mathcal{M}(A)$. 

\subsection{Acknowledgments}

I would like to thank Yanbo Fang for discussions, Jérôme Poineau for comments on the draft and Michael Temkin for answering questions about locally affinoid algebras.
I am indebted to the anonymous referee for many valuable suggestions and especially for pointing out several mistakes in the original version of the manuscript.
I would like to thank Marco Maculan for pointing out a mistake in the original \cref{thm:Liumoreq}.

\section{Preliminaries}\label{sec:pre}
Let $k$ be a complete non-Archimedean valued field.
\subsection{Analytic spaces}
In this paper, by a $k$-analytic space, we mean a $k$-analytic space in the sense of \cite{Berk93}. The category of $k$-analytic spaces is denoted by $\AnaCat_k$.
For each $k$-analytic space $X$, we endow $X$ with the G-topology as in \cite{Berk93}. The corresponding site is still denoted by $X$. There is a natural sheaf of rings $\mathcal{O}_{X_G}$ making $X$ a ringed site. We always omit the subindex G and write $\mathcal{O}_X$ instead. The category of coherent sheaves on $X$ is denoted by $\CohCat_X$.
 
Strict $k$-analytic spaces are defined as in \cite{Berk93}. Recall that by a celebrated result of Temkin \cite{Tem04}, strict $k$-analytic spaces form a full subcategory of the category of $k$-analytic spaces if $k$ is non-trivially valued. The category of $k$-affinoid spaces is denoted by $\AffCat_k$, see \cite{Berk12}. The category of $k$-affinoid algebras is denoted by $\AffAlgCat_k$. There is an equivalence between $\AffCat_k$ and $\AffAlgCat_k$, given by the functor of global sections $X\mapsto H^0(X,\mathcal{O}_X)$ and the functor of Berkovich spectrum $A\mapsto \Sp A$.

\subsection{A representability theorem}
The following result is analogous to \cite[Proposition~4.5.4]{EGAI}.
\begin{theorem}\label{thm:rep}
Let $F$ be a presheaf on $\AnaCat_k$. Assume that
\begin{enumerate}
	\item $F$ satisfies the sheaf property for the G-topology, namely, for any $k$-analytic space $X$, any G-covering $\{U_i\}$ of $X$, $F(X)$ is the equalizer of 
	\[
	\prod_i F(U_i)\rightrightarrows \prod_{i,j}F(U_i\cap U_j)\,.
	\]
	\item There is a family $\{F_i\}_i$ of subfunctors of $F$ such that 
	\begin{enumerate}
		\item Each $F_i$ is representable by a $k$-analytic space $X_i$.
		\item Each $F_i\rightarrow F$ is representable by a \emph{closed} (resp. \emph{open}) analytic domain. In particular, after base change to $X_j$, $F_i\rightarrow F$ is represented by a closed (resp. open) analytic domain $U_{ji}$. In the closed case, we assume furthermore that for each $i$, the collection of $j$ such that $U_{ij}\neq \emptyset$ is finite.
		\item The collection $F_i$ covers $F$.
	\end{enumerate}
\end{enumerate}
Then $F$ is representable.
\end{theorem}
\begin{proof}
Let $\xi_i\in F_i(X_i)$ be the universal family of the presheaf $F_i$. By assumption a morphism of $k$-analytic spaces $T\rightarrow X_i$ factors through $U_{ij}$ if{f} $\xi_i|_T\in F_j(T)$. In particular, $\xi_i|_{U_{ij}}\in F_j(U_{ij})$. So we get a morphism $f_{ij}:U_{ij}\rightarrow X_j$ such that $f_{ij}^*\xi_j=\xi_i|_{U_{ij}}$. By definition of $U_{ji}$, we know that $f_{ij}$ factors through $U_{ji}$. Now observe that $(f_{ij}\circ f_{ji})^*\xi_j=f_{ji}^*\xi_i=\xi_j$, we conclude that $f_{ij}\circ f_{ji}=\mathrm{id}_{U_{ji}}$. In particular, all $f_{ij}$ are isomorphisms. It is formal to see that the glueing conditions are satisfied by the $f_{ij}$'s, hence we can glue the $X_i$'s together to get a $k$-analytic space $X$ by \cite[Proposition~1.3.3]{Berk93}. It is formal to check that $X$ together with the glueing $\xi$ of $\xi_i$ represents $F$. We refer to \cite[\href{https://stacks.math.columbia.edu/tag/01JJ}{Tag 01JJ}]{stacks-project} for the omitted details.
\end{proof}

\subsection{Polyradii}

\begin{definition}
	A \emph{polyradius} is an element $r\in \mathbb{R}_{>0}^n$ for some $n\in \mathbb{N}$. A polyradius $r$ is \emph{$k$-free} if the components of $r$ are linearly independent as elements in the $\mathbb{Q}$-vector space $\mathbb{Q}\otimes_{\mathbb{Z}}(\mathbb{R}_{>0}/|k^*|)$.

	For any $k$-polyradius $r\in \mathbb{R}_{>0}^n$, define $k_r$ as the $k$-affinoid algebra of formal series 
	\[
		\left\{\,\sum_{\alpha\in \mathbb{Z}^n} a_{\alpha}T^{\alpha}\in k[[T_1,\ldots,T_n]] \mid a_{\alpha}\in k, |a_{\alpha}|r^{\alpha}\to 0 \text{ when } |\alpha|\to \infty  \,\right\}
	\]
	endowed with the multiplicative norm $\sum_{\alpha\in \mathbb{Z}^n} a_{\alpha}T^{\alpha}\mapsto \max_{\alpha\in \mathbb{Z}^n} a_{\alpha}T^{\alpha}$.

	When $r$ is $k$-free, $k_r$ is a field. 

	For a given $k$-free polyradius $r$, a given Banach $k$-algebra $A$, for any Banach $A$-module $M$, we write $A_r=A\hat{\otimes}_k k_r$, $M_r=M\hat{\otimes}_k k_r$. Note that $A_r$ is a Banach $k_r$-algebra and $M_r$ is a Banach $A_r$-module.
	Similarly, given any $k$-analytic space, write $X_r:=X\times_{\Sp k} \Sp k_r$.
\end{definition}

\subsection{The category of Banach modules}

We briefly summarize a few results in \cite{BBK17}. For the basic theory of quasi-Abelian categories, see \cite{Sch99}. 

Let $\BanCat_k$ be the category of Banach $k$-modules, where morphisms are bounded homomorphisms. Recall that $\BanCat_k$ is a closed symmetric  monoidal quasi-Abelian category with all finite limits and finite colimits, where the $\otimes$ operator is given by the completed tensor product $\hat{\otimes}$. Moreover, finite products and finite coproducts coincide. The category $\BanCat_k$ has enough projectives. All projective objects in $\BanCat_k$ are flat in the sense of \cite{BBB16}. We have derived categories $\mathrm{D}^*(\BanCat_k)$, where $*$ means $+$, $-$, $\mathrm{b}$ or empty. Let $\BanAlgCat_k$ be the category of Banach $k$-algebras, which is also the category of algebras in the symmetric monoidal category $\BanCat_k$ in the abstract sense.
Let $A\in \BanAlgCat_k$ be a Banach $k$-algebra. Let $\BanModCat_A$ be the category of Banach $A$-modules, which is also the category of $A$-modules in the symmetric monoidal category $\BanCat_k$ in the abstract sense. Recall that $\BanModCat_A$ is also a closed symmetric  monoidal quasi-Abelian category with all finite limits and finite colimits, where the $\otimes$ operator is also given by $\hat{\otimes}$.
We write $\mathrm{D}^*(A)=\mathrm{D}^*(\BanModCat_A)$.

\begin{definition}\label{def:trans}
Let $f:A\rightarrow B$ be a morphism in $\BanAlgCat_k$. Let $M\in \BanModCat_A$. We say that $M$ is \emph{transversal} to $f$ if the natural morphism
\[
	M\otL_A B\rightarrow M\hat{\otimes}_A B
\]
in $\mathrm{D}^-(A)$ is an isomorphism.
\end{definition}

\begin{proposition}[{\cite[Proposition~2.1.2]{Berk12}}]\label{prop:flatkr}
For any $k$-free polyradius, the Banach $k$-module $k_r$ is flat in $\BanCat_k$: for any admissible exact short sequence $0\rightarrow E\rightarrow F\rightarrow G\rightarrow 0$ in $\BanCat_k$,  the following sequence is also admissible and exact:
\[
0\rightarrow E_r\rightarrow F_r\rightarrow G_r\rightarrow 0\,.
\]
\end{proposition}

\section{Liu spaces and Liu algebras}\label{sec:Liusp}
Let $k$ be a complete non-Archimedean valued field.

\subsection{Liu spaces}
In this section, we recall the basic theory of Liu $k$-analytic spaces following \cite{MP21} and \cite{Liu90}.

\begin{definition}[{\cite[Definition~1.9]{MP21}}]\label{def:Liuspace}
A $k$-analytic space $X$ is called \emph{Liu} if 
\begin{enumerate}
	\item $X$ is quasi-compact, separated.
	\item $X$ is holomorphically separable: for any $x,y\in X$, $x\neq y$, there is $f\in H^0(X,\mathcal{O}_X)$ such that $|f(x)|\neq |f(y)|$.
	\item $\mathcal{O}_X$ is universally acyclic: for any analytic extension $k'/k$, $H^i(X_{k'},\mathcal{O}_{X,k'})=0$ for any $i>0$.
\end{enumerate}
A morphism of Liu $k$-analytic spaces is a morphism of the underlying $k$-analytic spaces. We denote the category of Liu $k$-analytic spaces by $\LiuCat_k$.
\end{definition}

\begin{example}
A $k$-affinoid space is a Liu $k$-analytic space. But the converse fails in general. We refer to \cite[Section~5]{Liu90} for details. In fact, the theory of non-Archimedean pinching in \cite{Tem21} gives plenty of such examples.
\end{example}

\begin{definition}\label{def:Liudom}
Let $X$ be a $k$-analytic space. An analytic domain $Z$ of $X$ is called a \emph{Liu domain} if $Z$ is a Liu $k$-analytic space.
\end{definition}

\begin{definition}Let $X$ be a $k$-analytic space.
We say $X$ is \emph{cohomologically Stein} if for any coherent sheaf of $\mathcal{O}_X$-modules $\mathcal{F}$, 
\[
	H^i(X,\mathcal{F})=0\,,\quad i>0\,.
\]
We say $X$ is \emph{universally cohomologically Stein} if for any analytic extension $k'/k$, $X_{k'}$ is cohomologically Stein.
\end{definition}

\begin{theorem}[{\cite[Theorem~1.11]{MP21}, \cite[Théorème~2]{Liu90}}]\label{thm:Liuchar}
Let $X$ be a separated, quasi-compact $k$-analytic space. Then the following are equivalent:
\begin{enumerate}
	\item $X$ is Liu.
	\item $X$ is universally cohomologically Stein.
	\item $X$ is holomorphically separable and $\mathcal{O}_X$ is universally acyclic.
\end{enumerate}
Moreover, if $k$ is non-trivially valued and $X$ is strict, then the conditions are equivalent to
\begin{enumerate}[resume]
	\item $X$ is rig-holomorphically separable and $\mathcal{O}_X$ is acyclic.
\end{enumerate}
\end{theorem}
Note that in (4), we only need acyclicity of $\mathcal{O}_X$ instead of universal acyclicity as explained in \cite{MP21}. 

For the definition of rig-holomorphically separability, we refer to \cite[Definition~1.5]{MP21}.

\begin{theorem}[{\cite[Corollary~1.16]{MP21}}]\label{thm:finiteLiu}
Let $f:Y\rightarrow X$ be a finite morphism of $k$-analytic spaces. Then 
\begin{enumerate}
	\item If $X$ is Liu, then so is $Y$.
	\item If $Y$ is Liu and $f$ is surjective, then $X$ is Liu.
\end{enumerate}
\end{theorem}
\begin{theorem}[{\cite[Corollary~1.15, Corollary~1.17]{MP21}}]
Let $X$ be a $k$-analytic space. Then
\begin{enumerate}
	\item For any analytic extension $k'/k$, $X_{k'}$ is Liu if{f} $X$ is Liu.
	\item Assume that $X$ is separated. Then $X$ is Liu if{f} $X^{\Red}$ is. 
	\item Assume that $X$ is separated. Then $X$ is Liu if{f} each irreducible component of $X$ is. 
\end{enumerate}
\end{theorem}
\begin{proof}
We only have to make the following remark to (1): $X$ is separated if{f} $X_{k'}$ is. This follows from \cite[Theorem~1.2]{CT19}.
\end{proof}

\begin{proposition}\label{prop:LiuCart}
Let $f:Y\rightarrow X$, $g:X'\rightarrow X$ be morphisms in $\LiuCat_k$. Then $Y':=Y\times_X X'\in \LiuCat_k$.
\end{proposition}
\begin{proof}
We have the following Cartesian diagram
\[
\begin{tikzcd}
Y' \arrow[d, "{(f,g)}", swap] \arrow[rd, phantom, "\square"] \arrow[r] & Y\times X' \arrow[d,"f\times g"] \\
X \arrow[r, "\Delta_X"]                & X\times X               
\end{tikzcd}\,.
\]
As $X$ is separated, $\Delta_X$ is a closed immersion, so is the morphism $Y'\rightarrow Y\times X'$. By \cref{thm:finiteLiu}, in order to show that $Y'$ is Liu, it suffices to show that $Y\times X'$ is Liu. This follows from \cite[Theorem~A.6]{MP21}.
\end{proof}

\begin{corollary}\label{cor:intLiu}
Let $X$ be a separated $k$-analytic space. Let $Y_1,Y_2$ be Liu domains in $X$, then $Y_1\cap Y_2$ is also a Liu domain.
\end{corollary}

\subsection{Liu algebras}

\begin{definition}\label{def:Liualge}
A \emph{Liu $k$-algebra} is a  Banach $k$-algebra $A$ such that there is a Liu $k$-analytic space such that $A\cong H^0(X,\mathcal{O}_X)$, where the isomorphism is an isomorphism of Banach $k$-algebras. A Liu $k$-algebra is said to be \emph{strict} if there is a strict Liu $k$-analytic space with $A\cong H^0(X,\mathcal{O}_X)$ in $\BanAlgCat_k$.

A morphism of Liu $k$-algebras is a bounded homomorphism of the underlying Banach $k$-algebras.

The category of Liu $k$-algebras is denoted by $\LiuAlgCat_k$. It is a full subcategory of $\BanAlgCat_k$.
\end{definition}

\begin{proposition}\label{prop:Liualgbasic}
Let $A$ be a Liu $k$-algebra. Then 
\begin{enumerate}
	\item $A$ is Noetherian and all of its ideals are closed.
	\item Suppose that $k$ is non-trivially valued and $A$ is strict. For any maximal ideal $\mathfrak{m}$ of $A$, $A/\mathfrak{m}$ is finite dimensional as a vector space over $k$.
	\item We have
	\[
		\bigcap_{\mathfrak{m}\in \Max(A)}\bigcap_{n=1}^{\infty} \mathfrak{m}^n=0\,.
	\]
\end{enumerate}
\end{proposition}
\begin{proof}
(1) That $A$ is noetherian follows from \cite[Proposition~2.6(3), Remark~2.7]{MP21}. When $k$ is non-trivially valued, all ideals are closed by \cite[Proposition~3.7.2.2]{BGR84}. In general, this follows from a base field extension argument, see \cite[Proposition~2.1.3]{Berk12}.

(2) By \cite[Proposition~1.3]{Liu90}, there is a rigid point $x\in X$ such that $\mathfrak{m}=\mathfrak{m}_{\Sp A,x}$. Take a strictly affinoid domain $\Sp B$ of $\Sp A$ containing $x$. Then $x$ is also rigid in $\Sp B$. It is well-known that $B/\mathfrak{m}_{\Sp B,x}$ is finite dimensional, hence so is $A/\mathfrak{m}$.
\iffalse
According to (1), $\mathfrak{m}$ is finitely generated, hence it induces a coherent sheaf $\widetilde{\mathfrak{m}}$ on $\Sp A$ (\cite[Lemma~2.4]{MP21}). By \cite[Proposition~2.6]{MP21},
\[
	H^0(X,\widetilde{\mathfrak{m}})=\mathfrak{m}\,.
\]
If $\mathfrak{m}$ is not equal to $\mathfrak{m}_{\Sp A,x}$ for any point $x\in \Sp A$, then $\widetilde{\mathfrak{m}}=\mathcal{O}_{\Sp A}$, hence $\mathfrak{m}=A$, which is a contradiction. This shows that there is $x\in \Sp A$ such that
\[
	\mathfrak{m}=\mathfrak{m}_{\Sp A,x}\,.
\]
In particular, $x$ lies in a strictly affinoid domain $\Sp B$ of $\Sp A$. It is well-known that $B/\mathfrak{m}_{\Sp B,x}$ is finite dimensional, hence so is $A/\mathfrak{m}$.
\fi

(3) Take an element $a\in A$ that lies in the intersection of all $\mathfrak{m}^n$ for any $\mathfrak{m}\in \Max A$, $n\geq 1$. Then By Krull's intersection theorem, for each $\mathfrak{m}\in \Max A$, there is an element $m\in \mathfrak{m}$ such that $(1-m)a=0$. Thus the annihilator of $a$ does not lie in any maximal ideal of $A$, hence $a=0$.
\end{proof}

\begin{corollary}\label{cor:conttoLiu}
Let $A$ be a Liu $k$-algebra. All $k$-algebra homomorphisms from a Banach $k$-algebra to $A$ are bounded. In particular, the Liu $k$-algebra structure of $A$ is uniquely determined by the underlying algebraic structure.
\end{corollary}
\begin{proof}
When $k$ is non-trivially valued and $A$ is strict, this follows from \cref{prop:Liualgbasic} and \cite[Proposition~3.7.5.2]{BGR84}.

In general, this follows from the change of base argument.
\end{proof}

\begin{theorem}[Liu]\label{thm:Liu}
The functor of global sections gives an anti-equivalence $\LiuCat_k\rightarrow \LiuAlgCat_k$. The inverse functor is denoted by $\Sp A$. Moreover, for any $k$-analytic space $Y$, any Liu $k$-analytic space $X$, the canonical map
\[
	\Hom_{\AnaCat_k}(Y,X)\rightarrow \Hom_{\AlgCat_k}(H^0(X,\mathcal{O}_X),H^0(Y,\mathcal{O}_Y))
\]
is bijective.
\end{theorem}
\begin{remark}
The space $\Sp A$ as a topological space coincides with the spectrum in the sense of Berkovich \cite[Section~1.2]{Berk12}. See \cite[Corollary~3.17]{MP21} for example.
\end{remark}

\begin{proof}
The latter statement is a formal consequence of the former.

When $k$ is non-trivially valued, by \cite[Proposition~3.2]{Liu90} and \cite[Theorem~1.6.1]{Berk93}, we know that the global section functor is an anti-equivalence from the category of strict Liu $k$-analytic spaces to the category of strict Liu $k$-algebras. 

In general, let $X$, $Y$ be Liu $k$-analytic spaces. Let $A=H^0(X,\mathcal{O}_X)$, $B=H^0(Y,\mathcal{O}_Y)$. Let $F:A\rightarrow B$ be a homomorphism of $k$-algebras. We want to construct a morphism $Y\rightarrow X$, whose induced map on global sections is given by $F$. We may assume that $Y$ is affinoid.
Take an analytic field extension $k'/k$, so that $k'$ is non-trivially valued, $A_{k'}$ and $B_{k'}$ become strict Liu $k$-algebras. We may assume that $k'=k_r$ for some $k$-free polyradius.
Then there is a unique morphism $g:Y_{k'}\rightarrow X_{k'}$ inducing $F_{k'}$. We claim that there is a unique morphism $f:Y\rightarrow X$ such that $g=f_{k'}$. Note that it is automatic that $f$ induces $F$ on global sections by \cite[Proposition~2.1.2]{Berk12}.

 By \cite[Proposition~3.13]{MP21}, there is a $k$-affinoid space $Z$, a locally closed immersion $h:X\rightarrow Z$ such that there is a finite covering $Z_1,\ldots,Z_m$ of $Z$ by rational domains such that $h^{-1}(Z_i)\rightarrow Z_i$ is a Runge immersion for each $i$:
\[
	\begin{tikzcd}
	Y_{k'} \arrow[r, "g"] \arrow[d]                            & X_{k'} \arrow[r, "h_{k'}"] \arrow[d] & Z_{k'} \arrow[d] \\
	Y \arrow[rr, "w"', bend right] \arrow[r, "f", dotted] & X \arrow[r, "h"]                & Z          
	\end{tikzcd}\,.
\]
Now observe that the composition of maps on global sections
\[
H^0(Z_{k'},\mathcal{O}_{Z_{k'}})\rightarrow H^0(X_{k'},\mathcal{O}_{X_{k'}})\rightarrow H^0(Y_{k'},\mathcal{O}_{Y_{k'}})
\]
is the same as the base extension of the map of $k$-algebras
\[
H^0(Z,\mathcal{O}_Z)\rightarrow A\xrightarrow{F} B\,.
\]
Thus if we denote by $w:Y\rightarrow Z$ the morphism of $k$-analytic spaces corresponding to this latter map, we have $w_{k'}=h_{k'}\circ g$.
Replacing $Y$ by $w^{-1}(Z_i)$, $X$ by $h^{-1}Z_i$ and $Z$ by $Z_i$ and applying \cite[Proposition~1.3.2]{Berk93} and \eqref{eq:h0tensor}, we may assume that $X\rightarrow Z$ is a Runge immersion. In particular $X$ is affinoid. We can take $f$ to be the morphism corresponding to $F$. Moreover, such $f$ (such that $g=f_{k'}$) is clearly unique. We conclude.
\end{proof}

\begin{lemma}\label{lma:Liutensor}
Let $A$ be a Liu $k$-algebra. Let $B$, $C$ be Liu $k$-algebras over $A$, then $B\hat{\otimes}_A C$ is Liu. In particular, for any $k$-free polyradius $r$, $A_r$ is a Liu $k$-algebra.
\end{lemma}
\begin{proof}
Let $Z=\Sp B\times_{\Sp A}\Sp C$.
By \cref{prop:LiuCart}, it suffices to prove
\begin{equation}\label{eq:h0tensor}
	H^0(Z,\mathcal{O}_Z)=B\hat{\otimes}_A C\,.
\end{equation}
Firstly, we consider the morphism 
\[
	\Delta_{\Sp A}:\Sp A\rightarrow \Sp A \times \Sp A\,.
\]
It is easy to see that this is a closed immersion, corresponding to the closed ideal $J$ in $A\hat{\otimes} A$ generated by $1\otimes a-a\otimes 1$ for $a\in A$. Also by \cite[Corollary~3.30]{PP}, we have 
\[
H^0(\Sp B\times \Sp C,\mathcal{O}_{\Sp B\times \Sp C})=B\hat{\otimes} C\,.
\]
\iffalse
When $\Sp C$ is affinoid, this follows from \cite[Corollary~3.17]{PP}. In general, 
the argument of \cite[Corollary~3.16]{PP} allows us to reduce to the previous case. To be more precise, 
let $\Sp C_i$ be a finite affinoid covering of $\Sp C$. Then we need to show that $B\hat{\otimes} C$ is the equalizer of $\prod_i B\hat{\otimes} C_i\rightrightarrows \prod_{i,j}B\hat{\otimes} C_{ij}$, where $\Sp C_{ij}=\Sp C_i\cap \Sp C_j$. It suffices to apply \cite[Proposition~3.21 and Remark~3.20]{PP}.
\fi

Hence the closed immersion $Z\rightarrow \Sp B\times\Sp C$ corresponds to the closed ideal of $B\hat{\otimes}C$ generated by $J(B\hat{\otimes} C)$. In particular, \eqref{eq:h0tensor} holds.
\end{proof}

\begin{definition}
Let $A$ be a Liu $k$-algebra. A Banach $A$-module $M$ is \emph{finite} if there is an admissible epimorphism $A^n\rightarrow M$.
\end{definition}
Let $\ModCat^{\mathrm{fin}}(A)$ be the category of finite $A$-modules.

\begin{proposition}\label{prop:finiteLiumod}
Let $A$ be a Liu $k$-algebra.
The forgetful functor from the category of finite Banach $A$-modules (with bounded $A$-algebra homomorphisms as morphisms) to $\ModCat^{\mathrm{fin}}(A)$ is an equivalence.
\end{proposition}
\begin{proof}
The functor is fully faithful. In fact, we prove more generally that for any finite Banach $A$-module $M$, any Banach $A$-module $N$, any $A$-linear map $F:M\rightarrow N$ is bounded. In fact, taking an admissible epimorphism $A^n\rightarrow M$, we may assume that $M=A^n$. In this case, the claim is clear.

The functor is essentially surjective. Take an $A$-linear epimorphism $\pi:A^n\rightarrow M$, then $\ker \pi$ is closed by \cref{prop:Liualgbasic} (1) and \cite[Proposition~3.7.2.2]{BGR84}, so we can endow $M$ with the residue Banach norm.
\end{proof}

\begin{proposition}\label{prop:finite}
Let $A$ be a Liu $k$-algebra. Let $r$ be a $k$-free polyradius. Let $M$ be  Banach $A$-module. Then $M$ is a finite Banach $A$-module if{f} $M_r$ is a finite Banach $A_r$-module.
\end{proposition}
\begin{proof}
This follows \emph{verbatim} from \cite[Proof of Proposition~2.1.11]{Berk12}.
\end{proof}

\begin{theorem}\label{thm:Day1}
Let $X=\Sp A$ be a Liu $k$-analytic space. Let $r$ be a $k$-free polyradius.
Consider a descent datum $(M_r,\varphi)$ of Banach modules over $A_r$. Then the descent datum is effective with respect to the natural morphism $\Sp A_r\rightarrow \Sp A$. Moreover, if $M_r$ is finitely generated as $A_r$-module, then the descent $M$ is finitely generated as $A$-module.
\end{theorem}
\begin{proof}
The first part follows \emph{verbatim} from \cite[Proof of Proposition~3.3]{Day21}. The second part follows from \cref{prop:finite}.
\end{proof}

\subsection{Coherent sheaves on Liu $k$-analytic spaces}

\begin{definition}
Let $X=\Sp A$ be a Liu $k$-analytic space.
Let $M$ be a finite $A$-module. Then we define a sheaf $\widetilde{M}$ on $X$ as the sheafification of the presheaf $\Sp B\mapsto M\otimes_A B$, where $\Sp B$ runs over the set of affinoid domains in $X$.
\end{definition}

\begin{proposition}[{\cite[Lemma~2.4]{MP21}}]
Let $X=\Sp A$ be a Liu $k$-analytic space. Let $M$ be a finite $A$-module. Then $\widetilde{M}$ is a coherent sheaf on $X$. Moreover, for each affinoid domain $\Sp B$ in $X$, 
\begin{equation}\label{eq:cohloc}
	H^0(\Sp B,\widetilde{M})=M\otimes_{A}B\,.
\end{equation}
\end{proposition}

Now we recall the theory of coherent sheaves on Liu $k$-analytic spaces. The following is the analogue of Cartan's Theorem~A.
\begin{theorem}[{\cite[Proposition~2.1]{MP21}}]\label{thm:CA}
Assume that $k$ is non-trivially valued.
Let $X$ be a Liu $k$-analytic space. For each coherent sheaf $\mathcal{F}$ on $X$ and each $x\in X$, $H^0(X,\mathcal{F})$ generates $\mathcal{F}_x$ as an $\mathcal{O}_{X,x}$-module. 
\end{theorem}

In the rigid setting, Cartan's Theorem~A and Theorem~B are due to Kiehl \cite{Kie67} and Tate \cite{Tat71} respectively.

As explained in \cite{Kie67}, Theorem~A and Theorem~B together imply the following result:
\begin{theorem}\label{thm:cohLiu}
Let $X=\Sp A$ be a Liu $k$-analytic space. Then the category of coherent sheaves on $X$ is equivalent to the category of finite $A$-modules. The functors are given by $\mathcal{F}\mapsto H^0(X,\mathcal{F})$ and $M\mapsto \widetilde{M}$ respectively. 
\end{theorem}
\begin{proof}
This result was proved in \cite[{Proposition~2.6}]{MP21} under the assumption that $k$ is non-trivially valued. When $k$ is trivially valued, take a $k$-free polyradius $r$ with at least one component.
By \cite[Théorème~3.13]{Day21}, the category of coherent sheaves on $X$ is equivalent to the category of descent data of coherent sheaves on $X_r$ with respect to $X_r\rightarrow X$. The latter category is equivalent to the category of descent data of finite $A_r$-modules with respect to $A\rightarrow A_r$, which is then equivalent to the category of finite $A$-modules by \cref{thm:Day1}. It is easy to see that the composition of these functors is exactly the one given in the theorem.
The functors in the proof are summarized in the following diagram:
\[
\begin{tikzcd}
\DesCat(\CohCat,X_r\rightarrow X) \arrow[d] \arrow[r] & \DesCat(\ModCat^{\mathrm{fin}},A\rightarrow A_r) \arrow[d] \\
\CohCat(X) \arrow[r]   & \ModCat^{\mathrm{fin}}(A)         
\end{tikzcd}\,.
\]
\end{proof}
In particular, \cref{thm:CA} holds even when $k$ is trivially valued.

\subsection{Quasi-coherent sheaves on Liu spaces}

\begin{definition}
    Let $f:A\rightarrow B$ be a morphism in $\LiuAlgCat_k$. We say $f$ is a \emph{homotopy epimorphism} if the corresponding morphism $\Sp B\rightarrow \Sp A$ of Liu $k$-spaces identifies $\Sp B$ with a Liu domain in $\Sp A$.
\end{definition}

\begin{definition}
Let $A$ be a Liu $k$-algebra. A Banach $A$-module $M$ is called \emph{transversal} if $M$ is transversal to all homotopy epimorphisms from $A$: for all homotopy epimorphism $A\rightarrow B$ to a Liu $k$-algebra $B$, the natural morphism
\[
	M\otL_A B\rightarrow M\hat{\otimes}_A B
\]
is an isomorphism.
\end{definition}

The following result will be proved in \cref{sec:BBK}.
\begin{theorem}\label{thm:liut}
Let $A$ be a Liu $k$-algebra. Let $B,C$ be Liu $k$-algebras over $A$ such that $\Sp C\rightarrow \Sp A$ is a Liu domain. Then the natural morphism
\[
	C\otL_A B \rightarrow C\hat{\otimes}_A B
\]
is an isomorphism.
\end{theorem}

\begin{definition}\label{def:qcoh}
Let $A$ be a Liu $k$-algebra. Let $M$ be a transversal Banach $A$-module. Write $X=\Sp A$.
We define a sheaf of $\mathcal{O}_X$-modules $\widetilde{M}$ as the sheafification of the presheaf
\[
\Sp B\mapsto M\hat{\otimes}_A B
\]
on $X$, where $\Sp B$ runs over the set of affinoid domains in $X$. We call $\widetilde{M}$ the sheaf associated to $M$.

An $\mathcal{O}_X$-module $\mathcal{M}$ is \emph{quasi-coherent} if there is a transversal $A$-module $M$ such that $\mathcal{M}=\widetilde{M}$.
\end{definition}

\begin{example}\label{ex:coharequas}
Let $X$ be a Liu $k$-analytic space. Then all coherent sheaves on $X$ are quasi-coherent. See for example \cite[Proof of Proposition~2.6(1)]{MP21}. To be more precise, the same proof shows that for any Liu domain $\Sp B\rightarrow \Sp A=X$, $B$ is a flat $A$-algebra. Let $M$ be a finite $A$-module.
Consider a presentation
\[
A^{\oplus S}\rightarrow A^{\oplus N}\rightarrow M\rightarrow 0\,.
\]
We have a commutative diagram with exact rows:
\[
\begin{tikzcd}
A^{\oplus S}\hat{\otimes}^{\mathbb{L}}_A B \arrow[r] \arrow[d] & A^{\oplus N}\hat{\otimes}^{\mathbb{L}}_A B \arrow[r] \arrow[d] & M\hat{\otimes}^{\mathbb{L}}_A B \arrow[r] \arrow[d] & 0 \\
A^{\oplus S}\otimes_A B \arrow[r]           & A^{\oplus N}\otimes_A B \arrow[r]           & M \otimes_A B \arrow[r]           & 0
\end{tikzcd}\,.
\]
In order to show that $M$ is transversal, it suffices to show that $A$ is, which is obvious.
\end{example}

\begin{theorem}[Tate acyclicity theorem]\label{thm:Tateacyc}
Let $X=\Sp A$ be a Liu $k$-analytic space. Let $\Sp A_1,\ldots,\Sp A_n$ be a finite G-covering of $X$ by Liu domains. Let $M$ be a transversal Banach $A$-module, then the following sequence is admissible and exact
\begin{equation}\label{eq:Tate}
	0\rightarrow M\rightarrow \prod_{i_1} M\hat{\otimes}_A A_{i_1}\rightarrow \prod_{i_1<i_2} M\hat{\otimes}_A A_{i_1}\hat{\otimes}_A A_{i_2}\rightarrow \cdots \rightarrow M\hat{\otimes}_A A_1\hat{\otimes}_A A_2\hat{\otimes}_A \cdots \hat{\otimes}_A A_n\rightarrow 0\,.
\end{equation}
\end{theorem}
\begin{proof}
It follows from the same proof as \cite[Lemma~5.34 and Remark~5.35]{BBK17}. We give a sketch for the convenience of the readers. When $M=A$, we can prove \eqref{eq:Tate} exactly as in the affinoid setting, namely it suffices to treat the case where the covering is given by $\{A\{ f\}, A\{ f^{-1}\}\}$ for some $f\in A$. Then the acyclicity follows from a direct computation.  See \cite[Chapter~8]{BGR84} for details. For a general $M$, taking derived tensor product with \eqref{eq:Tate} for $M=A$ and apply the transversality condition, we get \eqref{eq:Tate} for $M$.
\end{proof}

\begin{corollary}\label{cor:qcohLiusecliudom}
Let $X=\Sp A$ be a Liu $k$-analytic space. Let $\mathcal{M}$ be a quasi-coherent sheaf on $X$. Let $M=H^0(X,\mathcal{M})$. Then for any Liu domain $\Sp B$ in $X$, we have
\[
	H^0(\Sp B,\mathcal{M})=M\hat{\otimes}_A B\,.
\]
\end{corollary}

\begin{corollary}\label{cor:qcohvan}
Let $X=\Sp A$ be a Liu $k$-analytic space. Let $\mathcal{M}$ be a quasi-coherent sheaf on $X$. Then 
\[
	H^i(X,\mathcal{M})=0\,,\quad i>0\,.
\]
\end{corollary}
\begin{proof}
This follows from \cite[\href{https://stacks.math.columbia.edu/tag/01EW}{Tag 01EW}]{stacks-project}\footnote{This result is only stated for a ringed space, but it is easy to check that the proof works in the current situation as well.} and \cref{thm:Tateacyc}.
\end{proof}

\begin{definition}
	Let $X$ be a $k$-analytic space. Let $\mathcal{F}$ be a sheaf of $\mathcal{O}_X$-modules (resp. $\mathcal{O}_X$-algebras). A \emph{Banach structure} on $\mathcal{F}$ is the following data: given any Liu domain $\Sp A$ in $X$, $\mathcal{F}(\Sp A)$ is topologized so that it forms a Banach $A$-module (resp. Banach $A$-algebra). We assume that the following condition holds: if $\Sp A$, $\Sp B$ are Liu domains in $X$ such that $\Sp A$ is an analytic domain of $\Sp B$, then the natural morphism of $A$-modules (resp. $A$-algebras) $\mathcal{F}(\Sp B)\hat{\otimes}_B A\rightarrow \mathcal{F}(\Sp A)$ is bounded.

	An $\mathcal{O}_X$-module(resp. $\mathcal{O}_X$-algebra) with a given Banach structure is called a \emph{sheaf of Banach modules} (resp. \emph{sheaf of Banach algebras}) on $X$.

	A morphism $\mathcal{F}\rightarrow \mathcal{G}$ of sheaves of Banach modules (resp. sheaves of Banach algebras) on $X$ is a morphism of the underlying sheaves of modules (resp. sheaves of algebras) such that for each Liu domain $\Sp B$ in $X$, $\mathcal{F}(\Sp B)\rightarrow \mathcal{G}(\Sp B)$ is bounded.

	The category of sheaves of Banach modules on $X$ is denoted by $\BanModCat_X$.
\end{definition}

\begin{proposition}\label{prop:morcohtosheaf}
Let $X=\Sp A$ be a Liu $k$-analytic space. Let $\mathcal{M}$ be a quasi-coherent sheaf on $X$. Let $M=H^0(X,\mathcal{M})$. Let $\mathcal{F}$ be a sheaf of Banach $\mathcal{O}_X$-modules. Then
\[
	\Hom_{\BanModCat_X}(\mathcal{M},\mathcal{F})=\Hom_{\BanModCat_A}(M,H^0(X,\mathcal{F}))\,.
\]
\end{proposition}
\begin{proof}
Given a morphism $f:\mathcal{M}\rightarrow \mathcal{F}$ in $\BanModCat_X$, by taking global sections, we get $H^0(f):M\rightarrow H^0(X,\mathcal{F})$. Conversely, given a bounded homomorphism $F:M\rightarrow H^0(X,\mathcal{F})$, we construct the morphism of sheaves $f:\mathcal{M}\rightarrow \mathcal{F}$ as follows: for any affinoid domain $\Sp B$ in $X$, define $f(\Sp B):M\hat{\otimes}_A B\rightarrow H^0(\Sp B,\mathcal{F})$ as the natural homomorphism of Banach $B$-modules induced by the homomorphism of Banach $A$-modules: 
\[
	M\stackrel{F}{\longrightarrow} H^0(X,\mathcal{F})\rightarrow H^0(\Sp B,\mathcal{F})\,.
\]
By the obvious functoriality, this is a morphism of Banach $\mathcal{O}_X$-modules.
It is easy to verify that these maps are inverse to each other.
\end{proof}

\begin{theorem}\label{thm:qcohpush}
Let $f:\Sp B\rightarrow \Sp A$ be a morphism in $\LiuCat_k$. Let $\mathcal{M}$ be a quasi-coherent sheaf on $\Sp B$. Then $f_*\mathcal{M}$ is a quasi-coherent sheaf on $\Sp A$ associated to the transversal $A$-module $H^0(\Sp B,\mathcal{M})$.
\end{theorem}
\begin{proof}
Let $F:A\rightarrow B$ be the corresponding homomorphism of Liu $k$-algebras. Let $M=H^0(\Sp B,\mathcal{M})$. 
We claim that $M$ is transversal as Banach $A$-module.

This is proved in \cite[Lemma~4.48]{BBK17}, we reproduce the argument: let $\Sp D\rightarrow \Sp A$ be a Liu domain. We need to show that
\[
M\otL_A D=M\hat{\otimes}_A D\,.
\]
Observe that
\[
M\otL_A D=M\otL_B (B\otL_A D)=M\otL_B (B \hat{\otimes}_A D)=M \hat{\otimes}_B (B \hat{\otimes}_A D)=M\hat{\otimes}_A D\,,
\]
where for the second equality, we have applied \cref{thm:liut}; for the third we used \cref{lma:Liutensor} and the transversality of $M$. This concludes the claim.

In order to prove the theorem, it suffices to show $\widetilde{M^A}=f_*\mathcal{M}$. Here $M^A$ is $M$ regarded as a Banach $A$-module. To prove this, it suffices to take an affinoid domain $\Sp C$ in $\Sp A$ and show that
\begin{equation}\label{eq:qcohlocal}
	M\hat{\otimes}_A C= \mathcal{M}(f^{-1}\Sp C)\,.
\end{equation}
By \cref{lma:Liutensor}, $f^{-1}\Sp C$ is a Liu domain in $\Sp B$ and $f^{-1}\Sp C=\Sp (B\hat{\otimes}_A C)$.
Hence 
\eqref{eq:qcohlocal} follows from \cref{cor:qcohLiusecliudom}.
\end{proof}

\begin{lemma}\label{lma:trans}
Let $A$ be a Liu $k$-algebra. Consider an admissible exact sequence
\[
	0\rightarrow F\rightarrow G \rightarrow H
\] 
in $\BanModCat_A$.
Assume that $G$, $H$ are both transversal, then so is $F$.
\end{lemma}
This is clear by definition.

\begin{corollary}\label{cor:pushqcoh}
Let $f:Y\rightarrow X$ be a quasi-compact and quasi-separated morphism of $k$-analytic spaces. Assume that $X=\Sp A$ is Liu. Let $\mathcal{F}$ be a Banach sheaf of $\mathcal{O}_Y$-modules such that for each affinoid domain $\Sp C$ in $Y$, $\mathcal{F}|_{\Sp C}$ is quasi-coherent. 
Then $f_*\mathcal{F}$ is quasi-coherent on $X$.
\end{corollary}
\begin{proof}
Let $\{U_i=\Sp B_i\}$ be a finite affinoid covering of $Y$. For each $i,j$, let $U_{ij}=U_i\cap U_j$, take a finite affinoid covering $\{U_{ijk}\}$ of $U_{ij}$. Let $f_i$ (resp. $f_{ijk}$) be the restriction of $f$ to $U_i$ (resp. $U_{ijk}$). Then $f_{i*}\mathcal{F}$ (resp. $f_{ijk*}\mathcal{F}$) is the quasi-coherent sheaf associated to $\mathcal{F}(U_i)$ (resp. $\mathcal{F}(U_{ijk})$) by \cref{thm:qcohpush}. In particular, $\mathcal{F}(U_i)$ (resp. $\mathcal{F}(U_{ijk})$) is a transversal Banach $A$-module. 

There is an admissible exact sequence
\[
	0\rightarrow \mathcal{F}(Y)\rightarrow \prod_i\mathcal{F}(U_i)\rightarrow \prod_{i,j,k}\mathcal{F}(U_{ijk})\,.
\]
Thus $\mathcal{F}(Y)$ is a transversal Banach $A$-modules by \cref{lma:trans}. In particular, for any affinoid domain $\Sp B$ in $X$, we have an admissible exact sequence
\[
	0\rightarrow \mathcal{F}(Y)\hat{\otimes}_A B\rightarrow \prod_i\mathcal{F}(U_i)\hat{\otimes}_A B\rightarrow \prod_{i,j,k}\mathcal{F}(U_{ijk})\hat{\otimes}_A B\,.
\]
By our assumption and \cref{cor:qcohLiusecliudom}, this sequence can be rewritten as
\[
	0\rightarrow \mathcal{F}(Y)\hat{\otimes}_A B\rightarrow \prod_i\mathcal{F}(U_i\cap f^{-1}(\Sp B))\rightarrow \prod_{i,j,k}\mathcal{F}(U_{ijk}\cap f^{-1}(\Sp B))\,.
\]
It is now clear that $\widetilde{\mathcal{F}(Y)^A}=f_*\mathcal{F}$ and $f_*\mathcal{F}$ is quasi-coherent. 
\end{proof}

\begin{theorem}\label{thm:pullqcoh}
Let $f:Y=\Sp B\rightarrow X=\Sp A$ be a morphism in $\LiuCat_k$.
Let $\mathcal{F}$ be a quasi-coherent sheaf on $X$. Let $F=H^0(X,\mathcal{F})$. Assume that $\mathcal{F}$ is transversal to $f$:
\[
	F \otL_A C=F\hat{\otimes}_A C
\]
for all Liu domains $\Sp C$ in $\Sp B$.
Then the left adjoint $f^*$ of $f_*:\BanModCat_Y\rightarrow \BanModCat_X$ at $\mathcal{F}$ exists and $f^*\mathcal{F}$  is the quasi-coherent sheaf associated to $F\hat{\otimes}_A B$.
\end{theorem}
\begin{proof}
We claim that $F\hat{\otimes}_A B$ is a transversal Banach $B$-module. 

This is proved in \cite[Lemma~4.48]{BBK17}, we reproduce their proof: let $\Sp C\rightarrow \Sp B$ be a Liu domain, we need to show 
\[
(F\hat{\otimes}_A B)\otL_B C=(F\hat{\otimes}_A B)\hat{\otimes}_B C\,.
\]
In fact,
\[
(F\hat{\otimes}_A B)\otL_B C=(F\otL_A B)\otL_B C=F\otL_A C=F\hat{\otimes}_A C=(F\hat{\otimes}_A B)\hat{\otimes}_B C\,,
\]
which concludes the claim.

By \cref{prop:morcohtosheaf}, for any sheaf of Banach $\mathcal{O}_Y$-modules $\mathcal{G}$,
\[
	\Hom_{\BanModCat_Y}(\widetilde{F\hat{\otimes}_A B},\mathcal{G})=\Hom_{\BanModCat_B}(F\hat{\otimes}_A B,H^0(Y,\mathcal{G}))=\Hom_{\BanModCat_A}(F,H^0(Y,\mathcal{G}))\,.	
\]
On the other hand,  by \cref{prop:morcohtosheaf}, we have
\[
	\Hom_{\BanModCat_X}(\mathcal{F},f_*\mathcal{G})=\Hom_{\BanModCat_A}(F,H^0(Y,\mathcal{G}))\,.
\]
We conclude.
\end{proof}

\section{Liu morphisms and quasi-coherent sheaves of Liu algebras}\label{sec:Liumor}
Let $k$ be a complete non-Archimedean valued field.

\subsection{Liu morphisms}

\begin{definition}\label{def:Liumorp}
Let $f:X\rightarrow Y$ be a morphism in $\AnaCat_k$. We say $f$ is \emph{Liu} if for any Liu domain $Z$ of $Y$, $f^{-1}Z$ is a Liu domain. 

For any $k$-analytic space $Y$, let $\LiuMorCat_{\rightarrow Y,k}$ denote the category of Liu morphisms $X\rightarrow Y$. A morphism between two Liu morphisms $X_1\rightarrow Y$ and  $X_2\rightarrow Y$ is a morphism of in the over-category $\AnaCat_{k/Y}$.
\end{definition}

The following two propositions are obvious.
\begin{proposition}
Let $f:X\rightarrow Y$, $g:Y\rightarrow Z$ be morphisms in $\AnaCat_k$. Assume that $f$, $g$ are both Liu, then so is $g\circ f$.
\end{proposition}
\begin{proposition}
Let $f:X\rightarrow Y$ be a Liu morphism in $\AnaCat_k$. Then $f$ is separated and quasi-compact.
\end{proposition}

\begin{lemma}\label{lma:higherdirLiu}
Let $f:X\rightarrow Y$ be a morphism in $\LiuCat_k$. Let $\mathcal{F}$ be a coherent sheaf on $X$. Then $R^if_*\mathcal{F}=0$ for all $i>0$.
\end{lemma}
\begin{proof}
The problem is local, so it suffices to show that $H^i(f^{-1}(\Sp A),\mathcal{F})=0$ for any affinoid domain $\Sp A$ of $Y$. This follows from the fact that $f^{-1}(\Sp A)$ is Liu (\cref{prop:LiuCart}).
\end{proof}

\begin{theorem}\label{thm:Liumoreq} 
Let $f:X\rightarrow Y$ be a morphism in $\AnaCat_k$. Assume that $Y$ is Liu.
Then the following are equivalent:
\begin{enumerate}
	\item $f$ is Liu.
	\item $f$ is quasi-compact and separated, for any analytic field extension $k'/k$, and coherent sheaf $\mathcal{F}$ on $X_{k'}$, 
	\[
		R^if_{k'*}\mathcal{F}=0\,,\quad i>0\,.
	\]
	\item $X$ is Liu. 
\end{enumerate}
\end{theorem}
\begin{proof}
(1) $\implies$ (2):  We may assume that $k'=k$ and it suffices to prove that for any affinoid domain $\Sp A$ in $Y$, $H^i(f^{-1}(\Sp A),\mathcal{F})=0$ for all $i>0$, which is trivial as $f^{-1}(\Sp A)$ is Liu.

(2) $\implies$ (3): This follows from Leray's spectral sequence.

(3) $\implies$ (1): This follows from \cref{prop:LiuCart}.
\end{proof}

\begin{example}\label{ex:notGlocal}
Recall \cite[Définition~3.18]{Day21}: A morphism $f:X\rightarrow Y$ in $\AnaCat_k$ is said to be \emph{almost affinoid} (\emph{presque affinoïde} in French) if there is a G-covering of $Y$ by affinoid domains $\{U_i\}$ such that $f^{-1}U_i$ is affinoid for each $i$.

An almost affinoid morphism is not necessarily Liu even if the target is affinoid. See 
\cite[Example~9.1.2]{SW20} for a counterexample. I would like to thank Marco Maculan for pointing this out to me.
\end{example}

\subsection{Quasi-coherent sheaves of Liu algebras}

\begin{definition}\label{def:qcohshLiualg}
Let $X$ be a $k$-analytic space.  A sheaf of Banach algebras $\mathcal{F}$ on $X$ is a \emph{quasi-coherent sheaf of Liu $k$-algebras} if for each Liu domain $\Sp A$ in $X$, $H^0(\Sp A,\mathcal{F})$ is a Liu $k$-algebra and $\mathcal{F}|_{\Sp A}$ is a quasi-coherent sheaf (in the sense of \cref{def:qcoh}). A morphism of quasi-coherent sheaves of Liu $k$-algebras on $X$ is a homomorphism of the underlying sheaves of $\mathcal{O}_X$-algebras. We denote the category of quasi-coherent sheaves of Liu $k$-algebras on $X$ by $\QcohLiuAlgCat_{X,k}$.
\end{definition}
\begin{remark}
By \cref{cor:conttoLiu}, a sheaf of Liu $k$-algebras admits a natural Banach structure. Moreover, a morphism of quasi-coherent sheaves of Liu $k$-algebras on $X$ is automatically a morphism in $\BanModCat_X$. Hence $\QcohLiuAlgCat_{X,k}$ is a full subcategory of $\BanModCat_X$.
\end{remark}

\begin{remark}
We do not define a quasi-coherent sheaf on a $k$-analytic space. In fact, according to the philosophy of \cite{BBK17}, in the global setting, the correct notion to consider is the derived category of quasi-coherent sheaves. 
\end{remark}

\begin{proposition}\label{prop:rep}
Let $X$ be a \emph{separated} $k$-analytic space. Let $\mathcal{A}$ be a quasi-coherent sheaf of Liu $k$-algebras on $X$. Consider the presheaf $F$ on $\AnaCat_k$:
\[
	T\mapsto \left\{\, (f,\varphi): f\in \Hom_{\AnaCat_k}(T,X),\varphi\in\Hom_{\mathcal{O}_T}(f^*\mathcal{A},\mathcal{O}_T) \,\right\}\,.
\]
Then $F$ is representable.
\end{proposition}
\begin{proof}
Assume first that $X$ is paracompact.
It suffices to verify that the conditions of \cref{thm:rep} are satisfied. 

(1) The sheaf condition follows from \cite[Proposition~1.3.2]{Berk93}.

(2) 
Take a locally finite affinoid covering $\{U_i\}$ of $X$. Observe that each $U_i$ is closed as $X$ is separated. Take $F_i$ to be the subfunctor of $F$ consisting of pairs $(f:T\rightarrow S,\varphi)$ such that $f(T)\subseteq U_i$. 
Then $F_i$ is represented by $\Sp \mathcal{A}(U_i)$. Thus 2(a) is satisfied. The conditions 2(b) and 2(c) follows from the choice of $U_i$.

In general, take a paracompact open covering $\{V_i\}$ of $X$ as in the final step of \cite[Proof of Proposition~1.4.1]{Berk93}. 
Repeat the same construction as in the previous step, with $\{V_i\}$ in place of $\{U_i\}$, we get subfunctors $F_i$ of $F$. 
Again, it suffices to verify the conditions of 2(a), 2(b), 2(c) of \cref{thm:rep}. The conditions 2(b), 2(c) follows from the choice of $\{V_i\}$, while the condition 2(a) follows from the special we just treated. 
\end{proof}
\begin{remark}
Of course, in \cref{prop:rep}, one can weaken the separateness assumption to Hausdorff condition. It is not clear to the author if one can remove this condition.
\end{remark}

\begin{definition}\label{def:relspec}
Let $X$ be a \emph{separated} $k$-analytic space. Let $\mathcal{A}$ be a quasi-coherent sheaf of Liu $k$-algebras on $X$. We define the \emph{relative spectrum} $\Specrel_X \mathcal{A}$ as the $k$-analytic space representing the presheaf $F$ in \cref{prop:rep}. Note that there is a natural morphism $\pi:\Specrel_X \mathcal{A}\rightarrow X$. We sometimes call $\pi$ the relative spectrum as well.
\end{definition}
\begin{proposition}
Let $X$ be a separated $k$-analytic space. Let $\mathcal{A}$ be a quasi-coherent sheaf of Liu $k$-algebras on $X$. Let $\pi:\Specrel_X \mathcal{A}\rightarrow X$ be the relative spectrum, then
\begin{enumerate}
	\item For each Liu domain $\Sp A$ in $X$, the restriction of $\pi$ to $\pi^{-1}(\Sp A)\rightarrow \Sp A$ is the same as $\Sp H^0(\Sp A,\mathcal{A})\rightarrow \Sp A$.
	\item For any morphism of separated $k$-analytic spaces $g:X'\rightarrow X$, $g^*\mathcal{A}$ is a quasi-coherent sheaf of Liu $k$-algebras and the natural morphism
	\[
	X'\times_X \Specrel_{X}\mathcal{A}\rightarrow \Specrel_{X'} g^*\mathcal{A}
	\]
	is an isomorphism over $X'$.
	\item The universal map 
	\[
	\mathcal{A}\rightarrow \pi_*\mathcal{O}_{\Specrel_X \mathcal{A}}
	\]
	is an isomorphism of sheaves of Banach algebras on $X$.
\end{enumerate}
\end{proposition}
We omit the straightforward proof. See \cite[\href{https://stacks.math.columbia.edu/tag/01LQ}{Tag 01LQ}]{stacks-project} for example.
\begin{corollary}\label{cor:Liuequivalence}
Let $X$ be a separated $k$-analytic space. Then the functor 
\[
	\Specrel_X: \QcohLiuAlgCat_{X,k} \rightarrow \LiuMorCat_{\rightarrow X,k}
\]
is an anti-equivalence of categories. The quasi-inverse is given by $f\mapsto f_*$.
\end{corollary}

\section{Quasi-Liu morphisms}\label{sec:quasiLiu}

Let $k$ be a complete non-Archimedean valued field.

\begin{definition}\label{def:quasiLiuspace}
A $k$-analytic space $X$ is called \emph{quasi-Liu} if the following conditions hold:
\begin{enumerate}
	\item $X$ is quasi-compact.
	\item $H^0(X,\mathcal{O}_X)$ is a Liu $k$-algebra.
	\item There is a Liu $k$-analytic space $\Sp B$ and a morphism $i:X\rightarrow \Sp B$, which realizes $X$ as an analytic domain in $\Sp B$. 
\end{enumerate}
\end{definition}

\begin{proposition}\label{prop:quasiLiuglobal}
Let $X$ be a quasi-Liu $k$-analytic space. Then the natural morphism $X\rightarrow \Sp H^0(X,\mathcal{O}_X)$ is an analytic domain embedding.
\end{proposition}
\begin{proof}
Let $Y=\Sp B$ be a Liu $k$-analytic space such that there is a morphism $i:X\rightarrow Y$, which is an analytic domain embedding.  Now we have a natural homomorphism $B\rightarrow A$ given by the restriction map $B=H^0(Y,\mathcal{O}_Y)\rightarrow A=H^0(X,\mathcal{O}_X)$. In particular, we get a factorization $X\rightarrow \Sp A \rightarrow Y$ of $i$ by \cref{thm:Liu}. Now it remains to show that $X\rightarrow \Sp A$ is an analytic domain. Take $x\in X$. We can find rational domains $V_1,\ldots,V_m$ of $Y$ contained in $X$ such that $x\in \cap_i V_i$ and $\cup_i V_i$ is a neighborhood of $x$ in $Y$. Let $U_i$ be the rational domain of $\Sp A$ induced by $V_i$. We claim that $U_i\subseteq X$. Assuming this claim, then we find that $x\in \cup_i U_i$ and $\cup_i U_i\subseteq X$ is a neighborhood of $x$ in $\Sp A$. We conclude that $X\rightarrow \Sp A$ is indeed an analytic domain.

To prove the claim, we will fix some $i$ and omit the indices from $V_i$, $U_i$.
We write $V=\Sp B\{r^{-1}f/g\}$,
where $f=(f_1,\ldots,f_n)$ is a tuple of elements in $B$, $r=(r_1,\ldots,r_n)$ is a tuple of positive real numbers and $g$ is an element in $B$ such that $f_j$, $g$ do not have a common zero. Then $U=\Sp A\{r^{-1}f/g\}$. Let $X'$ denote the analytic domain of $X$ consisting of points where $|f_j|\leq r_j|g|$ for all $j=1,\ldots,n$. As $V\subseteq X$, we could identify $X'$ with the analytic domain in $Y$ defined by the same inequalities. In particular, $X'$ is a Liu space.
Take a finite affinoid covering $\Sp A_i$ of $X$, we know that $A$ is the equalizer of $\prod_i A_i \rightrightarrows \prod_{i,j} A_{ij}$, where $\Sp A_{ij}=\Sp A_i\cap \Sp A_j$. By \cref{thm:liut}, $A\{r^{-1}f/g\}$ is the equalizer of $\prod_i A_i\{r^{-1}f/g\} \rightrightarrows \prod_{i,j} A_{ij}\{r^{-1}f/g\}$. As $\Sp A_i\{r^{-1}f/g\}$ is an affinoid covering of $X'$, we find an isomorphism $H^0(X',\mathcal{O}_{X'})\cong A\{r^{-1}f/g\}$. It induces an isomorphism
 $X'\rightarrow U$ by \cref{thm:Liu}, which is the inverse of the composition $U\rightarrow V\rightarrow X'$. In particular, we find that $U\rightarrow X$ is injective.
\end{proof}

\begin{lemma}\label{lma:quasiLiubase1}
Let $f:X\rightarrow Y$ be a morphism in $\AnaCat_k$. Assume that $Y$ is Liu and $X$ is quasi-Liu. Let $g:Y'\rightarrow Y$ be a Liu domain in $Y$. 
 Then $X':=X\times_Y Y'$ is also quasi-Liu.
\end{lemma}
\begin{proof}
Let $f':X'\rightarrow Y'$ be the base change of $f$.
It suffices to show that $H^0(X',\mathcal{O}_{X'})$ is a Liu $k$-algebra. By decomposing $X\rightarrow Y$ as in the proof of \cref{prop:quasiLiuglobal}, we have the commutative diagram:
\[
\begin{tikzcd}
X' \arrow[r] \arrow[d] & X \arrow[d] \\
Y'\times_Y \Sp H^0(X,\mathcal{O}_X) \arrow[d] \arrow[r] & \Sp H^0(X,\mathcal{O}_X) \arrow[d] \\
Y' \arrow[r]           & Y          
\end{tikzcd}\,.
\]
Replacing $Y$ by $\Sp H^0(X,\mathcal{O}_X)$ and $Y'$ by $Y'\times_Y \Sp H^0(X,\mathcal{O}_X)$, we may assume that $H^0(X,\mathcal{O}_{X})=H^0(Y,\mathcal{O}_Y)$ and $f$ is the analytic domain embedding $X\rightarrow H^0(X,\mathcal{O}_X)$ in \cref{prop:quasiLiuglobal}. 

We have the following commutative diagram:
\[
\begin{tikzcd}
X' \arrow[r,"g'"] \arrow[d,"f'"] \arrow[rd,"\square",phantom] & X \arrow[d,"f"] \\
Y' \arrow[r,"g"]                      & Y         
\end{tikzcd}\,.
\]
Take a finite affinoid G-covering $X_i$ of $X$, then we get an admissible exact sequence
\[
0\rightarrow H^0(Y,\mathcal{O}_Y)\rightarrow \prod_i H^0(X_i,\mathcal{O}_X)\rightarrow \prod_{i,j} H^0(X_{ij},\mathcal{O}_X)\,,
\]
where $X_{ij}:=X_i\cap X_j$. Taking the derived tensor $\otL_{H^0(Y,\mathcal{O}_Y)}H^0(Y',\mathcal{O}_{Y'})$ and applying \cref{thm:liut} and \eqref{eq:h0tensor}, we get an admissible exact sequence
\[
0\rightarrow H^0(Y',\mathcal{O}_{Y'})\rightarrow \prod_i H^0(g'^{-1}(X_i),\mathcal{O}_{X'})\rightarrow \prod_{i,j} H^0(g'^{-1}(X_{ij}),\mathcal{O}_{X'})\,.
\]
In particular,
\[
H^0(Y',\mathcal{O}_{Y'})=H^0(X',\mathcal{O}_{X'})
\]
and this algebra is a Liu algebra.
Also observe that the morphism $f':X'\rightarrow Y'$ satisfies the assumption of \cref{def:quasiLiuspace}(3) and $X'$ is quasi-Liu.
\end{proof}

\begin{definition}\label{def:quasiLiumor}
Let $f:X\rightarrow Y$ be a morphism of $k$-analytic spaces. We say $f$ is \emph{quasi-Liu} if for any Liu domain $Z$ in $Y$, $f^{-1}Z$ is quasi-Liu.
\end{definition}

\begin{proposition}
Let $f:X\rightarrow Y$ be a quasi-Liu morphism in $\AnaCat_k$. Then $f$ is separated and quasi-compact.
\end{proposition}
This is obvious.

\begin{proposition}\label{prop:quasi-Liu}
Let $f:X\rightarrow Y$ be a morphism of $k$-analytic spaces. Assume that $Y$ is separated.
The following are equivalent:
\begin{enumerate}
	\item $f$ is quasi-Liu.
	\item $f_*\mathcal{O}_X$ is a quasi-coherent sheaf of Liu $k$-algebras and the natural morphism $X\rightarrow \Specrel_Y f_*\mathcal{O}_X$ is quasi-compact and realizes $X$ as an analytic domain.
	\item $f_*\mathcal{O}_X$ is a quasi-coherent sheaf of Liu $k$-algebras on $Y$ and $X$ can be realized as an analytic domain in $\Specrel_Y \mathcal{A}$ through a quasi-compact morphism $X\rightarrow \Specrel_Y \mathcal{A}$ over $Y$, where $\mathcal{A}$ is a quasi-coherent sheaf of Liu $k$-algebras on $Y$.
\end{enumerate}
\end{proposition}
\begin{proof}
It is clear that (2) $\implies$ (3) $\implies$ (1).

(1) $\implies$ (2): 
Observe that $f_*\mathcal{O}_X$ is quasi-coherent by \cref{cor:pushqcoh}. It is a quasi-coherent sheaf of Liu $k$-algebras by \cref{lma:quasiLiubase1}. The last assertion follows from \cref{prop:quasiLiuglobal}.
\end{proof}

\iffalse
\begin{corollary}
Let $f:X\rightarrow Y$ be a quasi-Liu morphism in $\AnaCat_k$. Let $g:Y'\rightarrow Y$ be a morphism in $\AnaCat_k$. Assume that $Y$ and $Y'$ are separated. 
Let $f':X'\rightarrow Y'$ be the base change of $X$ by $g$. Then $f'$ is also quasi-Liu.
\end{corollary}
\fi

\begin{proposition}\label{prop:compqL}
Let $f:X\rightarrow Y$, $g:Y\rightarrow Z$ be morphisms in $\AnaCat_k$.
If $f$ is quasi-Liu and $g$ is Liu, then $g\circ f$ is quasi-Liu.
\end{proposition}
\begin{proof}
We need to show that the inverse image of a Liu domain $U$ in $Z$ by $g\circ f$ is quasi-Liu. But $g^{-1}(U)$ is Liu and we find that $f^{-1}(g^{-1}(U))$ is quasi-Liu by definition.
\end{proof}

\section{Open problems}\label{sec:open}
Let $k$ be a complete non-Archimedean valued field.

We give a list of unsolved problems related to Liu $k$-algebras and Liu morphisms.

\begin{question}
Is there a global version of Zariski's main theorem in non-Archimedean geometry?
\end{question}
A local version is proved by Ducros in \cite[Théorème~3.2]{Duc07} based on Temkin's graded reduction. This theorem roughly says that a quasi-finite morphism of separated $k$-analytic spaces can be written locally as the composition of an étale morphism, an analytic domain embedding and a finite morphism. This theorem, however, does not tell us much information about the global structure of a quasi-finite morphism, in contrast to the classical Zariski's main theorem (\cite[\href{https://stacks.math.columbia.edu/tag/02LR}{Tag 02LR}]{stacks-project}). 

We would like to know if the following holds: 
\begin{conjecture}
Let $f:X\rightarrow S$ be a quasi-finite morphism of quasi-compacted, separated $k$-analytic spaces. Then we can decompose $f$ into $h\circ i\circ g$, where $g:X\rightarrow Y$ is finite, $i:Y\rightarrow Z$ is a quasi-compact analytic domain embedding, $h:Z\rightarrow S$ is étale.
\end{conjecture}
We hope to find suitable extra conditions on $f$, which guarantee that $i$ is a Liu domain embedding as well. 

\begin{question}
Are Liu $k$-algebras excellent?
\end{question}
In the case of affinoid algebras, this is proved by Ducros \cite{Duc09}. The author is not sure if Ducros' argument can be generalized to the current setting.

\begin{question}\label{ques:liudesc}
Can Liu morphisms be effectively descended with respect to fpqc (or Tate-flat) coverings?
\end{question}
In a previous version of this paper, the author claimed a proof. But as pointed out by the referee, the proof contains a gap. By \cite[Théorème~A]{Day21}, the essential difficulty is to treat the case of descending along a finite faithfully flat morphism of affinoid spaces.

\appendix 
\section{Results from Ben-Bassat--Kremnizer}\label{sec:BBK}

We slightly generalize a few results in \cite{BBK17}.

\begin{definition}\label{def:homepi}
Let $f:A\rightarrow B$ be a morphism in $\BanAlgCat_k$. We say $f$ is a \emph{homotopy epimorphism} if the following equivalent conditions are satisfied
\begin{enumerate}
	\item $\mathbb{L}f_*:\mathrm{D}^-(B)\rightarrow \mathrm{D}^-(A)$ is fully faithful.
	\item The natural morphism 
	\[
		\mathbb{L}f^*\circ \mathbb{L}f_*\rightarrow \mathrm{id}_{\mathrm{D}^-(B)}
	\] 
	is a natural equivalence.
	\item $B\otL_A B =B$.
\end{enumerate}
\end{definition}

\begin{definition}
Let $f:\Sp A\rightarrow \Sp B$ be a morphism in $\LiuCat_k$. We say $f$ is a \emph{homotopy monomorphism} if the corresponding morphism $B\rightarrow A$ in $\LiuAlgCat_k$ is a homotopy epimorphism (\cref{def:homepi}). 
\end{definition}

\begin{lemma}
Let $A\rightarrow B$ be a morphism in $\LiuAlgCat_k$. For any $r>0$, $f\in A$, we have the natural isomorphisms in $\mathrm{D}^{-}(A)$:
\[
	B\otL_A A\{ r^{-1}f \}\rightarrow B\hat{\otimes}_A A\{ r^{-1}f \}\,,\quad B\otL_A A\{ rf^{-1} \}\rightarrow B\hat{\otimes}_A A\{ rf^{-1} \}\,.
\]
\end{lemma}
\begin{proof}
We only treat the former. As in the case of affinoid algebras (\cite[Lemma~5.13]{BBK17}), it suffices to prove that the morphism
\[
	T-f: A\{ r^{-1}f \} \rightarrow A\{ r^{-1}f \}
\]
is a strict monomorphism. That this morphism is a monomorphism is well-known (and can be proved exactly as in the affinoid case).

To see $T-f$ is strict, by \cite[Proposition~2.1.2]{Berk12}, we could assume that $k$ is non-trivially valued. Then the image of $T-f$ is closed by \cref{prop:Liualgbasic}. Hence $T-f$ is strict.
\end{proof}
\begin{lemma}
Let $A\rightarrow B$ be a morphism in $\LiuAlgCat_k$. Let $f_1,\ldots,f_n,g\in A$ be elements that generate $A$. Let $r_1,\ldots,r_n\in \mathbb{R}_{>0}$, Then we have the natural isomorphism in $\mathrm{D}^{-}(A)$:
\[
	B\otL_A A\{ r_i^{-1}f_i/g \}\rightarrow B\hat{\otimes}_A A\{ r_i^{-1}f_i/g \}\,.
\]
\end{lemma}
The proof goes exactly as \cite[Lemma~5.14]{BBK17}.

\begin{lemma}
Let $A$ be a Liu $k$-algebra. Let $A_1$, $A_2$, $B$ be Liu $k$-algebras over $A$.  Assume that
\begin{enumerate}
	\item $\Sp A_i\rightarrow \Sp A$ ($i=1,2$) are Liu domains.
	\item $\Sp A_1\cup \Sp A_2$ is also a Liu domain in $\Sp A$ with Liu $k$-algebra $C$.
	\item Let $A_0$ be the Liu $k$-algebra of the Liu domain $\Sp A_1\cap \Sp A_2$ (c.f. \cref{cor:intLiu}). Then the following natural morphisms are isomorphisms
	\[
		A_i \otL_A B\rightarrow A_i \hat{\otimes}_A B
	\]
	for $i=0,1,2$.
\end{enumerate}
Then we have a natural isomorphism
\[
	C\otL_A B\rightarrow C\hat{\otimes}_A B\,.
\]
\end{lemma}
This is obvious.
\begin{theorem}
Let $A$ be a Liu $k$-algebra. Let $B,C$ be Liu $k$-algebras over $A$ such that $\Sp C\rightarrow \Sp A$ is a Liu domain. Then we have a natural isomorphism
\[
	C\otL_A B \rightarrow C\hat{\otimes}_A B\,.
\]
In particular, $\Sp C\rightarrow \Sp A$ is a homotopy monomorphism.
\end{theorem}
\begin{proof}
Having established the three preceding lemmas, the proof is the same as \cite[Proof of Theorem~5.16]{BBK17}.
\end{proof}
\begin{theorem}
Let $f:A\rightarrow B$ be a morphism in $\LiuAlgCat_k$. Then $f$ is a homotopy epimorphism if{f} the corresponding morphism $\Sp B\rightarrow \Sp A$ is a Liu domain.
\end{theorem}
\begin{proof}
Same proof as \cite[Theorem~5.31]{BBK17}.
\end{proof}
In terms of \cite{BBK17}, we have shown that $\LiuAlgCat_k$ is a homotopy Zariski transversal subcategory of $\BanCat_k$.

\newpage
\printbibliography

\bigskip
  \footnotesize

  Mingchen Xia, \textsc{Department of Mathematics, Chalmers Tekniska Högskola, G\"oteborg}\par\nopagebreak
  \textit{Email address}, \texttt{xiam@chalmers.se}\par\nopagebreak
  \textit{Homepage}, \url{http://www.math.chalmers.se/~xiam/}.
\end{document}